\documentclass{amsart}
\usepackage{amssymb,amsmath,amsthm}
\usepackage[dvips]{graphicx}

\let\oldmarginpar\marginpar
\renewcommand\marginpar[1]{\-\oldmarginpar[\raggedleft\footnotesize #1]%
{\raggedright\footnotesize #1}}

\begin{document}

\newtheorem{theorem}{Theorem}[section]
\newtheorem{corollary}[theorem]{Corollary}
\newtheorem{lemma}[theorem]{Lemma}
\newtheorem{proposition}[theorem]{Proposition}
\theoremstyle{definition}
\newtheorem{definition}[theorem]{Definition}
\theoremstyle{remark}
\newtheorem{remark}[theorem]{Remark}
\theoremstyle{definition}
\newtheorem{example}[theorem]{Example}

\def\rank{{\text{rank}\,}}

\numberwithin{equation}{section}

\title[Inequalities for the Casorati curvatures]{Inequalities for the Casorati curvatures of real hypersurfaces in some Grassmannians}

\author{Kwang-Soon Park}
\address{Division of General Mathematics, Room 4-107, Changgong Hall, University of Seoul, Seoul 02504, Republic of Korea}
\email{parkksn@gmail.com}
%\thanks{The research of the first author was supported by the National Research Foundation of Korea (NRF)
%grant funded by the Korea government (MEST)(No. 2009-0057445).}


   %Primary 57M50; Secondary 57N16, 53A20, 53C15}

\begin{abstract}
In this paper we obtain two types of optimal inequalities consisting
of the normalized scalar curvature and the generalized normalized
$\delta$-Casorati curvatures for real hypersurfaces of complex
two-plane Grassmannians and complex hyperbolic two-plane
Grassmannians. We also find the conditions on which the equalities
hold.
\end{abstract}

\maketitle
\section{Introduction}\label{intro}
\addcontentsline{toc}{section}{Introduction}

In 1968, S. S. Chern \cite{C} gave an open question, which deals
with the existence of minimal immersions into any Euclidean spaces.

In 1993, to solve such problems, B. Y. Chen \cite{C1} introduced the
notion of Chen invariants (or $\delta$-invariants) and obtained some
optimal inequalities consisting of intrinsic invariants and some
extrinsic invariants for any Riemannian submanifolds. It was the
starting point of the theory of Chen invariants, which are one of
the most interesting topics in differential geometry (\cite{ACM},
\cite{C2}, \cite{CDVV}, \cite{OM}, \cite{V}).

The Casorati curvature of a submanifold in a Riemannian manifold is
the extrinsic invariant given by the normalized square of the second
fundamental form and some optimal inequalities containing Casorati
curvatures were obtained  for submanifolds of real space forms,
complex space forms, and quaternionic space forms (\cite{DHV},
\cite{G}, \cite{SSV}, \cite{LV}). The notion of Casorati curvature
is the extended version of the notion of the principal curvatures of
a hypersurface of a Riemannian manifold. So, it is both important
and very interesting to obtain some optimal inequalities for the
Casorati curvatures of submanifolds in any ambient Riemannian
manifolds.

For the real hypersurfaces of both complex space forms and
quaternionic space forms, it is well-known that there don't exist
any real hyersurfaces with paralell shape operator by the Codazzi
equation.

We also know the following.  A real hypersurface of a complex
projective space with a parallel second fundamental form is locally
congruent to a tube over some totally geodesic complex submanifold
with some radius \cite{KM}.

There don't exist any real Hopf hypersurface with parallel Ricci
tensor of a complex projective space \cite{K}.

A real hypersurface of a quaternionic projective space with the
shape operator to be parallel with respect to some almost contact
structure vector fields is locally congruent to a tube over some
quaternionic projective space with some radius \cite{P}.

Since such results had been introduced, many geometers studied real
hypersurfaces of a complex two-plane Grassmannian
$G_2(\mathbb{C}^{m+2})$.

Some natural two distributions of a real hypersurface of
$G_2(\mathbb{C}^{m+2})$ with $m\geq 3$ are invariant under the shape
operator if and only if either it is an open part of a tube around a
totally geodesic submanifold $G_2(\mathbb{C}^{m+1})$ of
$G_2(\mathbb{C}^{m+2})$ or it  is an open part of a tube around a
totally geodesic submanifold $\mathbb{H}P^n$ of
$G_2(\mathbb{C}^{m+2})$ \cite{BS0}.

There don't exist any real hypersurfaces of $G_2(\mathbb{C}^{m+1})$
with parallel second fundamental form \cite{S}.

As we know, both a complex two-plane Grassmannian
$G_2(\mathbb{C}^{m+2})$ and a complex hyperbolic two-plane
Grassmannian $SU_{2,m}/S(U_2\cdot U_m)$ are examples of Hermitian
symmetric spaces with rank 2. And studying a real hypersurface of
Hermitian symmetric spaces with rank 2 is very important and one of
the main topics in submanifold theory. And the classification of
real hypersurfaces of Hermitian symmetric spaces with rank 2 is one
of the important subjects in differential geometry.

Many geometers obtained some results on $SU_{2,m}/S(U_2\cdot U_m)$.

The maximal complex subbundle and the maximal quaternionic subbundle
of a real hypersurface of $SU_{2,m}/S(U_2\cdot U_m)$ are invariant
under the shape operator if and only if it is locally congruent to
an open part of some particlar type of hypersurfaces \cite{BS}.

There does not exist any real hypersurface in complex hyperbolic
two-plane Grassmannian $SU_{2,m}/S(U_2\cdot U_m)$, $m\geq 3$, with
commuting shape operator \cite{PSW}.

There does not exist any Hopf hypersurface in complex hyperbolic
two-plane Grassmannian $SU_{2,m}/S(U_2\cdot U_m)$, $m\geq 3$, with
commuting shape operator on the complex maximal subbundle
\cite{PSW}.

As the author knows, there are only examples for such optimal
inequalities at the submanifolds of constant space forms (i.e., real
space forms, complex space forms, and quaternionic space forms).
Therefore, the optimal inequalities, which are given here, are both
meaningful and very important.

\section{Preliminaries}\label{prelim}

In this section we remind some notions, which are used in the
following sections.

Given an almost Hermitian manifold $(N,g,J)$, i.e., $N$ is a
$C^{\infty}$-manifold, $g$ is a Riemannian metric on $N$, and $J$ is
a compatible almost complex structure on $(N,g)$ (i.e., $J\in
\text{End}(TN)$, $J^2=-id$, $g(JX,JY)=g(X,Y)$ for any vector fields
$X,Y\in \Gamma(TN)$), we call the manifold $(N,g,J)$ {\em
K\"{a}hler} if $\nabla J=0$, where $\nabla$ is the Levi-Civita
connection of $g$.

Let $N$ be a $4m-$dimensional $C^{\infty}$-manifold and let $E$ be a
rank 3 subbundle of $\text{End} (TN)$ such that for any point $p\in
N$ with a neighborhood $U$, there exists a local basis $\{
J_1,J_2,J_3 \}$ of sections of $E$ on $U$ satisfying for all
$\alpha\in \{ 1,2,3 \}$
$$
J_{\alpha}^2=-id, \quad
J_{\alpha}J_{\alpha+1}=-J_{\alpha+1}J_{\alpha}=J_{\alpha+2},
$$
where the indices are taken from $\{ 1,2,3 \}$ modulo 3.

Then we call $E$ an {\em almost quaternionic structure} on $N$ and
$(N,E)$ an {\em almost quaternionic manifold} \cite{AM}.

Moreover, let $g$ be a Riemannian metric on $N$ such that for any
point $p\in N$ with a neighborhood $U$, there exists a local basis
$\{ J_1,J_2,J_3 \}$ of sections of $E$ on $U$ satisfying for all
$\alpha\in \{ 1,2,3 \}$
\begin{equation}\label{hypercom}
J_{\alpha}^2=-id, \quad
J_{\alpha}J_{\alpha+1}=-J_{\alpha+1}J_{\alpha}=J_{\alpha+2},
\end{equation}
\begin{equation}\label{hypermet}
g(J_{\alpha}X, J_{\alpha}Y)=g(X, Y)
\end{equation}
for all vector fields  $X, Y\in \Gamma(TN)$, where the indices are
taken from $\{ 1,2,3 \}$ modulo 3.

Then we call $(N,E,g)$ an {\em almost quaternionic Hermitian
manifold} \cite{IMV}.

For convenience, the above basis $\{ J_1,J_2,J_3 \}$ satisfying
(\ref{hypercom}) and (\ref{hypermet}) is said to be a {\em
quaternionic Hermitian basis}.

Let $(N,E,g)$ be an almost quaternionic Hermitian manifold.

We call $(N,E,g)$ a {\em quaternionic K\"{a}hler manifold} if there
exist locally defined 1-forms $\omega_1, \omega_2, \omega_3$ such
that for $\alpha \in \{ 1,2,3 \}$
$$
\nabla_X J_{\alpha} =
\omega_{\alpha+2}(X)J_{\alpha+1}-\omega_{\alpha+1}(X)J_{\alpha+2}
$$
for any vector field $X\in \Gamma(TN)$, where the indices are taken
from $\{ 1,2,3 \}$ modulo 3 \cite{IMV}.

If there exists a global parallel quaternionic Hermitian basis $\{
J_1,J_2,J_3 \}$ of sections of $E$ on $N$ (i.e., $\nabla J_{\alpha}
= 0$ for $\alpha \in \{ 1,2,3 \}$, where $\nabla$ is the Levi-Civita
connection of the metric $g$), then $(N, E, g )$ is said to be a
{\em hyperk\"{a}hler manifold}. Furthermore, we call $(J_1, J_2,
J_3, g )$ a {\em hyperk\"{a}hler structure} on $N$ and $g$ a {\em
hyperk\"{a}hler metric} \cite{B}.

Let $G_2(\mathbb{C}^{m+2})$ be the set of all complex 2-dimensional
linear subspaces of $\mathbb{C}^{m+2}$. Then we know that the
complex two-plane Grassmannian $G_2(\mathbb{C}^{m+2})$ has some
Riemannian symmetric structure (\cite{B1}, \cite{S}). Denote by $g$
the corresponding metric. As we know, it is the unique compact
irreducible Riemannian manifold such that it has both a K\"{a}hler
structure $J$ and a quaternionic K\"{a}hler structure $E$ with
$J\notin E$. And $G_2(\mathbb{C}^{m+2})$ is the unique compact
irreducible K\"{a}hler quaternionic K\"{a}hler manifold such that it
is not a hyperk\"{a}hler manifold.

Given a local quaternionic Hermitian basis $\{J_1,J_2,J_3\}$ of $E$,
we have
\begin{equation}\label{comm}
J_i\circ J = J\circ J_i
\end{equation}
for $J_i\in \{J_1,J_2,J_3\}$ and the Riemannian curvature tensor
$\overline{R}$ of $(G_2(\mathbb{C}^{m+2}),g)$ is locally given by
\begin{eqnarray}
\quad \overline{R}(X,Y)Z   &=& g(Y,Z)X-g(X,Z)Y            \label{curv1} \\
&+& g(JY,Z)JX-g(JX,Z)JY-2g(JX,Y)JZ   \nonumber  \\
&+& \sum_{\alpha=1}^{3} \{ g(J_{\alpha}Y,Z)J_{\alpha}X-g(J_{\alpha}X,Z)J_{\alpha}Y-2g(J_{\alpha}X,Y)J_{\alpha}Z \}  \nonumber  \\
&+& \sum_{\alpha=1}^{3} \{
g(J_{\alpha}JY,Z)J_{\alpha}JX-g(J_{\alpha}JX,Z)J_{\alpha}JY \}
\nonumber
\end{eqnarray}
for any vector fields $X,Y,Z\in \Gamma(TG_2(\mathbb{C}^{m+2}))$
(\cite{B1}, \cite{S}).

Similarly, let $SU_{2,m}/S(U_2\cdot U_m)$ be the set of all complex
two-dimensional linear subspaces in indefinite complex Euclidean
space $\mathbb{C}_2^{m+2}$. Then the complex hyperbolic two-plane
Grassmannian $SU_{2,m}/S(U_2\cdot U_m)$ becomes a connected simply
connected irreducible Riemannian symmetric space with noncompact
type and rank two \cite{BS}. Denote by $g$ the corresponding metric.
And it is the unique noncompact irreducible manifold with negative
scalar curvature such that it has a K\"{a}hler structure $J$ and a
quaternionic K\"{a}hler structure $E$ with $J\notin E$ \cite{BS}.

We also know that given a local quaternionic Hermitian basis
$\{J_1,J_2,J_3\}$ of $E$, we have
\begin{equation}\label{comm2}
J_i\circ J = J\circ J_i
\end{equation}
for $J_i\in \{J_1,J_2,J_3\}$ and the Riemannian curvature tensor
$\overline{R}$ of $(SU_{2,m}/S(U_2\cdot U_m),g)$ is locally given by
\begin{eqnarray}
\quad \overline{R}(X,Y)Z  &=& -\frac{1}{2} [ g(Y,Z)X-g(X,Z)Y            \label{curv2} \\
&+& g(JY,Z)JX-g(JX,Z)JY-2g(JX,Y)JZ   \nonumber  \\
&+& \sum_{\alpha=1}^{3} \{ g(J_{\alpha}Y,Z)J_{\alpha}X-g(J_{\alpha}X,Z)J_{\alpha}Y-2g(J_{\alpha}X,Y)J_{\alpha}Z \}  \nonumber  \\
&+& \sum_{\alpha=1}^{3} \{
g(J_{\alpha}JY,Z)J_{\alpha}JX-g(J_{\alpha}JX,Z)J_{\alpha}JY \}
\nonumber ]
\end{eqnarray}
for any vector fields $X,Y,Z\in \Gamma(TSU_{2,m}/S(U_2\cdot U_m))$
\cite{BS}.

Furthermore, we remind some notions, which are used later. Let
$(N,g_N)$ be a Riemannian manifold and $M$ a submanifold of
$(N,g_N)$ with the induced metric $g_M$. Then the {\em Gauss} and
{\em Weingarten formula} are given by
\begin{equation}\label{eq3}
\overline{\nabla}_X Y = \nabla_X Y + h(X,Y)
\end{equation}
for $X,Y\in \Gamma(TM)$,
\begin{equation}\label{eq4}
\overline{\nabla}_X N = -A_N X + \nabla_X^{\perp} N
\end{equation}
for $X\in \Gamma(TM)$ and $N\in \Gamma(TM^{\perp})$, where
$\overline{\nabla}$ and $\nabla$ are the Levi-Civita connections of
the metrics $g_N$ and $g_M$, respectively, $h$ is the second
fundamental form of $M$ in $N$, $A$ is the shape operator of $M$ in
$N$, and $\nabla^{\perp}$ is the normal connection of $M$ in $N$.

We denote by $\overline{R}$ and $R$ the Riemannian curvature tensors
of $g_N$ and $g_M$, respectively. Then the {\em Gauss equation} is
given by
\begin{eqnarray}
R(X,Y,Z,W)  &=& \overline{R}(X,Y,Z,W)           \label{eq5} \\
&+& g_N(h(X,W), h(Y,Z)) - g_N(h(X,Z), h(Y,W))   \nonumber
\end{eqnarray}
for any vector fields $X,Y,Z,W\in \Gamma(TM)$, where
$\overline{R}(X,Y,Z,W):= g_N(\overline{R}(X,Y)Z, W)$ and
$R(X,Y,Z,W):= g_M(R(X,Y)Z, W)$.

Consider a local orthonormal tangent frame $\{ e_1, \cdots, e_m \}$
of the tangent bundle $TM$ of $M$ and a local orthonormal normal
frame $\{ e_{m+1}, \cdots, e_n \}$ of the normal bundle $TM^{\perp}$
of $M$ in $N$. The {\em scalar curvature} $\tau$ of $M$ is defined
by
\begin{equation}\label{eq6}
\tau = \sum_{1\leq i<j\leq m} K(e_i\wedge e_j),
\end{equation}
where $K(e_i\wedge e_j) := R(e_i,e_j,e_j,e_i)$ for $1\leq i<j\leq
m$. And the {\em normalized scalar curvature} $\rho$ of $M$ is given
by
\begin{equation}\label{eq7}
\rho = \frac{2\tau}{m(m-1)}.
\end{equation}
We denote by $H$ the {\em mean curvature vector field} of $M$ in
$N$. i.e., $H = \frac{1}{m}\sum_{i=1}^m h(e_i,e_i)$. Conveniently,
let $h_{ij}^{\alpha} := g_N(h(e_i,e_j), e_{\alpha})$ for $i,j\in \{
1, \cdots, m \}$ and $\alpha\in \{ m+1, \cdots, n \}$. Then we have
the {\em squared mean curvature} $||H||^2$ of $M$ in $N$ and the
{\em squared norm} $||h||^2$ of $h$ as follows:
\begin{equation}\label{eq8}
||H||^2 = \frac{1}{m^2}\sum_{\alpha =m+1}^n (\sum_{i=1}^m
h_{ii}^{\alpha})^2,
\end{equation}
\begin{equation}\label{eq9}
||h||^2 = \sum_{\alpha =m+1}^n \sum_{i,j=1}^m (h_{ij}^{\alpha})^2.
\end{equation}
The {\em Casorati curvature} $C$ of $M$ in $N$ is defined by
\begin{equation}\label{eq10}
C := \frac{1}{m} ||h||^2.
\end{equation}
The submanifold $M$ is said to be {\em invariantly quasi-umbilical}
if there exists a local orthonormal normal frame $\{ e_{m+1},
\cdots, e_n \}$ of $M$ in $N$ such that the shape operators
$A_{e_{\alpha}}$ have an eigenvalue of multiplicity $m-1$ for all
$\alpha\in \{ m+1, \cdots, n \}$ and the distinguished
eigendirection of $A_{e_{\alpha}}$ is the same for each $\alpha\in
\{ m+1, \cdots, n \}$ \cite{B2}.

Let $L$ be a $k$-dimensional subspace of $T_p M$, $k\geq 2$, for
$p\in M$ such that $\{ e_1, \cdots, e_k \}$ is an orthonormal basis
of $L$. Then the {\em scalar curvature} $\tau (L)$ of the $k$-plane
$L$ is given by
\begin{equation}\label{eq11}
\tau (L) := \sum_{1\leq i<j\leq k} K(e_i\wedge e_j)
\end{equation}
and the {\em Casorati curvature} $C(L)$ of the subspace $L$ is
defined by
\begin{equation}\label{eq12}
C(L) := \frac{1}{k} \sum_{\alpha =m+1}^n \sum_{i,j=1}^k
(h_{ij}^{\alpha})^2.
\end{equation}
The {\em normalized $\delta$-Casorati curvatures} $\delta_c (m-1)$
and $\widehat{\delta}_c (m-1)$ of $M$ in $N$ are given by
\begin{equation}\label{eq13}
[\delta_c (m-1)](p) := \frac{1}{2}C(p) + \frac{m+1}{2m} \inf \{
C(L)|L \ \text{a hyperplane of} \ T_p M \}
\end{equation}
\begin{equation}\label{eq14}
[\widehat{\delta}_c (m-1)](p) := 2C(p) - \frac{2m-1}{2m} \sup \{
C(L)|L \ \text{a hyperplane of} \ T_p M \}.
\end{equation}
We define the {\em generalized normalized $\delta$-Casorati
curvatures} $\delta_c (r,m-1)$ and $\widehat{\delta}_c (r,m-1)$ of
$M$ in $N$ as follows:
\begin{eqnarray}
&& [\delta_c (r,m-1)](p) := rC(p)           \label{eq15} \\
&& + \frac{(m-1)(m+r)(m^2-m-r)}{rm} \inf \{ C(L)|L \ \text{a
hyperplane of} \ T_p M \}   \nonumber
\end{eqnarray}
for $0<r<m^2-m$,
\begin{eqnarray}
&& [\widehat{\delta}_c (r,m-1)](p) := rC(p)           \label{eq16} \\
&& - \frac{(m-1)(m+r)(r-m^2+m)}{rm} \sup \{ C(L)|L \ \text{a
hyperplane of} \ T_p M \}   \nonumber
\end{eqnarray}
for $r>m^2-m$.

Notice that $[\delta_c (\frac{m(m-1)}{2},m-1)](p) = m(m-1)[\delta_c
(m-1)](p)$ and $[\widehat{\delta}_c (2m(m-1),m-1)](p) =
m(m-1)[\widehat{\delta}_c (m-1)](p)$ for $p\in M$ so that the
generalized normalized $\delta$-Casorati curvatures $\delta_c
(r,m-1)$ and $\widehat{\delta}_c (r,m-1)$ are the generalized
versions of the normalized $\delta$-Casorati curvatures $\delta_c
(m-1)$ and $\widehat{\delta}_c (m-1)$, respectively.

Throughout this paper, we will use the above notations.

\section{Some optimal inequalities}\label{semi}

In this section we will obtain some optimal inequalities consisting
of the normalized scalar curvature and the generalized normalized
$\delta$-Casorati curvatures for real hypersurfaces of complex
two-plane Grassmannians and complex hyperbolic two-plane
Grassmannians.

\begin{theorem}
Let $M$ be a real hypersurface of a complex two-plane Grassmannians
$G_2(\mathbb{C}^{m+2})$ with $n=4m-1$. Then we have

(a) The generalized normalized $\delta$-Casorati curvature $\delta_c
(r,n-1)$ satisfies
\begin{equation}\label{eq17}
\rho \leq \frac{\delta_c (r,n-1)}{n(n-1)} + \frac{n+9}{n}
\end{equation}
for any $r\in \mathbb{R}$ with $0<r<n(n-1)$.

(b) The generalized normalized $\delta$-Casorati curvature
$\widehat{\delta}_c (r,n-1)$ satisfies
\begin{equation}\label{eq18}
\rho \leq \frac{\widehat{\delta}_c (r,n-1)}{n(n-1)} + \frac{n+9}{n}
\end{equation}
for any $r\in \mathbb{R}$ with $r>n(n-1)$.

Moreover, the equalities hold in the relations (\ref{eq17}) and
(\ref{eq18}) if and only if $M$ is an invariantly quasi-umbilical
submanifold with flat normal connection in $G_2(\mathbb{C}^{m+2})$
such that  with some orthonormal tangent frame $\{ e_1, \cdots, e_n
\}$ of $TM$ and orthonormal normal frame $\{ e_{n+1}=e \}$ of
$TM^{\perp}$, the shape operator $A_e$ takes the following form
\begin{equation}\label{eq188}
A_e = \left(
        \begin{array}{ccccc}
          a & 0 & \cdots & 0 & 0 \\
          0 & a & \cdots & 0 & 0 \\
          \vdots & \vdots & \ddots & \vdots & \vdots \\
          0 & 0 & \cdots & a & 0 \\
          0 & 0 & \cdots & 0 & \frac{n(n-1)}{r}a \\
        \end{array}
      \right).
\end{equation}
\end{theorem}

\begin{proof}
Since $M$ is a real hypersurface of $G_2(\mathbb{C}^{m+2})$ with a
unit normal vector field $e$, we may choose a local orthonormal
tangent frame $\{ e_1, \cdots, e_n \}$ of $TM$ and an orthonormal
normal frame $\{ e_{n+1}=e \}$ of $TM^{\perp}$ such that
\begin{align*}
&e_{m+i} = J_1 e_i,   \\
&e_{2m+i} = J_2 e_i,   \\
&e_{3m+i} = J_3 e_i ,  \\
&e_{4m-3} = \xi_1 = -J_1 e,  \\
&e_{4m-2} = \xi_2 = -J_2 e,  \\
&e_{4m-1} = e_n = \xi_3 = -J_3 e
\end{align*}
for $1\leq i\leq m-1$, where $\{ J_1,J_2,J_3 \}$ is a local
quaternionic Hermitian basis of $E$.

Let $\xi := -Je$.

Using (\ref{curv1}) and (\ref{eq5}), we get
\begin{align*}
2\tau(p) &= n(n-1) + 3\sum_{i,j=1}^n g(e_i,Je_j)^2   \\
&+ \sum_{\alpha=1}^3\sum_{i,j=1}^n \{ 3g(e_i,J_{\alpha}e_j)^2 +
g(e_i,J_{\alpha}Je_i)\cdot g(e_j,J_{\alpha}Je_j) -
g(e_i,J_{\alpha}Je_j)^2 \}     \\
&+ n^2||H||^2 - ||h||^2 \quad \text{for} \ p\in M.
\end{align*}
With some computations, we obtain
\begin{align*}
&3\sum_{i,j=1}^n g(e_i,Je_j)^2+ \sum_{\alpha=1}^3\sum_{i,j=1}^n \{
3g(e_i,J_{\alpha}e_j)^2 + g(e_i,J_{\alpha}Je_i)\cdot
g(e_j,J_{\alpha}Je_j) - g(e_i,J_{\alpha}Je_j)^2 \}   \\
&=3\sum_{i,j=1}^n g(e_i,Je_j)^2 + \sum_{\alpha=1}^3 \{
3(n-1)+(\sum_{i=1}^n g(e_i,J_{\alpha}Je_i))^2     \\
&-\sum_{i,j=1}^n g(e_i,Je_j)^2+\sum_{j=1}^n
g(\xi_{\alpha},Je_j)^2-\sum_{j=1}^n g(e,Je_j)^2 \}   \\
&=9(n-1) - 3\sum_{j=1}^n g(e,Je_j)^2 + \sum_{\alpha=1}^3 \{
(\sum_{i=1}^n g(e_i,J_{\alpha}Je_i))^2 + \sum_{j=1}^n
g(\xi_{\alpha},Je_j)^2 \}.
\end{align*}
Moreover,
\begin{align*}
\sum_{j=1}^n g(e,Je_j)^2
&= \sum_{j=1}^n g(\xi,e_j)^2 =  \sum_{j=1}^{n+1} g(\xi,e_j)^2 \\
&=||\xi||^2 = g(e,e) = 1,
\end{align*}
\begin{align*}
&\sum_{i=1}^n g(e_i,J_{\alpha}Je_i)   \\
&= -\sum_{i=1}^n g(J_{\alpha}e_i,Je_i) \\
&=-\sum_{i=1}^{m-1} \{ g(J_{\alpha}e_i,Je_i) +
g(J_{\alpha}J_1e_i,JJ_1e_i) + g(J_{\alpha}J_2e_i,JJ_2e_i) +
g(J_{\alpha}J_3e_i,JJ_3e_i) \}     \\
&-(g(J_{\alpha}\xi_1,J\xi_1) + g(J_{\alpha}\xi_2,J\xi_2) +
g(J_{\alpha}\xi_3,J\xi_3))   \\
&=0-(g(J_{\alpha}J_1e,JJ_1e) + g(J_{\alpha}J_2e,JJ_2e) +
g(J_{\alpha}J_3e,JJ_3e))  \ (\text{by} \ (\ref{comm})) \\
&=g(J_{\alpha}e,Je) = g(\xi_{\alpha},\xi)  \ (\text{by} \
(\ref{comm})),
\end{align*}
\begin{align*}
\sum_{j=1}^n g(\xi_{\alpha},Je_j)^2
&= \sum_{j=1}^n g(J\xi_{\alpha},e_j)^2   \\
&= \sum_{j=1}^{n+1} g(J\xi_{\alpha},e_j)^2 - g(J\xi_{\alpha},e)^2 \\
&= ||J\xi_{\alpha}||^2 - g(\xi_{\alpha},\xi)^2   \\
&= 1 - g(\xi_{\alpha},\xi)^2.
\end{align*}
Hence,
\begin{align*}
&9(n-1) - 3\sum_{j=1}^n g(e,Je_j)^2 + \sum_{\alpha=1}^3 \{
(\sum_{i=1}^n g(e_i,J_{\alpha}Je_i))^2 + \sum_{j=1}^n
g(\xi_{\alpha},Je_j)^2 \}   \\
&= 9(n-1).
\end{align*}
Therefore,
\begin{equation}\label{eq19}
2\tau(p) = (n+9)(n-1)+ n^2||H||^2 - nC.
\end{equation}
Conveniently, let $h_{ij} := h_{ij}^{n+1} = g(h(e_i,e_j),e_{n+1})$
for $i,j\in \{ 1,2,\cdots,n \}$.

Consider the quadratic polynomial in the components of the second
fundamental form
\begin{equation}\label{eq20}
\mathcal{P} := rC +
\frac{(n-1)(n+r)(n^2-n-r)}{rn}C(L)-2\tau(p)+(n+9)(n-1),
\end{equation}
where $L$ is a hyperplane of $T_p M$.

Now, we will deal with some linear algebraic properties of the
quadratic polynomial $\mathcal{P}$. Without loss of generality, we
may assume that $L$ is spanned by $e_1, \cdots,e_{n-1}$.

With a simple calculation, by (\ref{eq19}), we have
\begin{eqnarray}
\quad \mathcal{P}  &=& \frac{r}{n} \sum_{i,j=1}^n h_{ij}^2 + \frac{(n+r)(n^2-n-r)}{rn} \sum_{i,j=1}^{n-1} h_{ij}^2 -2\tau(p)+ (n+9)(n-1) \label{eq21} \\
&=& \frac{n+r}{n}\sum_{i,j=1}^n h_{ij}^2 + \frac{(n+r)(n^2-n-r)}{rn} \sum_{i,j=1}^{n-1} h_{ij}^2- (\sum_{i=1}^n h_{ii})^2  \nonumber  \\
&=& \sum_{i=1}^{n-1}[\frac{n^2+n(r-1)-2r}{r}h_{ii}^2 + \frac{(n+r)}{n} (h_{in}^2+h_{ni}^2)]  \nonumber  \\
&+&  \frac{(n+r)(n-1)}{r} \sum_{1\leq i\neq j\leq n-1} h_{ij}^2 -
\sum_{1\leq i\neq j\leq n} h_{ii}h_{jj} + \frac{r}{n} h_{nn}^2.
\nonumber
\end{eqnarray}
From (\ref{eq21}), the critical points $h^c =
(h_{11},h_{12},\cdots,h_{nn})$ of $\mathcal{P}$ are the solutions of
the system of linear homogeneous equations:
\begin{equation}\label{eq22}
\left\{
  \begin{array}{ll}
    \frac{\partial \mathcal{P}}{\partial h_{ii}} &= \frac{2(n+r)(n-1)}{r}h_{ii} -2\sum_{k=1}^n h_{kk} =0 \\
    \frac{\partial \mathcal{P}}{\partial h_{nn}} &= \frac{2r}{n} h_{nn} -2\sum_{k=1}^{n-1} h_{kk} =0 \\
    \frac{\partial \mathcal{P}}{\partial h_{ij}} &= \frac{2(n+r)(n-1)}{r}h_{ij} =0 \\
    \frac{\partial \mathcal{P}}{\partial h_{in}} &= \frac{2(n+r)}{n}h_{in} =0   \\
    \frac{\partial \mathcal{P}}{\partial h_{ni}} &= \frac{2(n+r)}{n}h_{ni} =0
  \end{array}
\right.
\end{equation}
for $i,j\in \{ 1,2,\cdots,n-1 \}$ with $i\neq j$.

From (\ref{eq22}), any solutions $h^c$ satisty $h_{ij}=0$ for
$i,j\in \{ 1,2,\cdots,n \}$ with $i\neq j$.

Moreover, we get the Hessian matrix $\mathcal{H}(\mathcal{P})$ of
$\mathcal{P}$ as follows:
$$
\mathcal{H}(\mathcal{P}) =
\left(
  \begin{array}{ccc}
    H_1 & 0 & 0 \\
    0 & H_2 & 0 \\
    0 & 0 & H_3 \\
  \end{array}
\right),
$$
where
$$
H_1 =
\left(
  \begin{array}{ccccc}
    \frac{2(n+r)(n-1)}{r}-2 & -2 & \cdots & -2 & -2 \\
    -2 & \frac{2(n+r)(n-1)}{r}-2 & \cdots & -2 & -2 \\
    \vdots & \vdots & \ddots & \vdots & \vdots \\
    -2 & -2 & \cdots & \frac{2(n+r)(n-1)}{r}-2 & -2 \\
    -2 & -2 & \cdots & -2 & \frac{2r}{n} \\
  \end{array}
\right),
$$
$0$ denotes the zero matrices with the corresponding sizes, and the
diagonal matrices $H_2$, $H_3$ are given by
$$
H_2 =
\text{diag}(\frac{2(n+r)(n-1)}{r},\frac{2(n+r)(n-1)}{r},\cdots,\frac{2(n+r)(n-1)}{r}),
$$
$$
H_3 =
\text{diag}(\frac{2(n+r)}{n},\frac{2(n+r)}{n},\cdots,\frac{2(n+r)}{n}).
$$
Then we can find that the Hessian matrix $\mathcal{H}(\mathcal{P})$
has the following eigenvalues
\begin{align*}
&\lambda_{11}=0, \lambda_{22}= \frac{2(n^3-n^2+r^2)}{rn}, \lambda_{33}= \cdots = \lambda_{nn}= \frac{2(n+r)(n-1)}{r},  \\
&\lambda_{ij}= \frac{2(n+r)(n-1)}{r}, \lambda_{in}= \lambda_{ni}=
\frac{2(n+r)}{n}
\end{align*}
for $i,j\in \{ 1,2,\cdots,n-1 \}$ with $i\neq j$.

Hence, we know that $\mathcal{P}$ is parabolic and has a minimum
$\mathcal{P}(h^c)$ at any solution $h^c$ of the system (\ref{eq22}).
Applying (\ref{eq22}) to (\ref{eq21}), we obtain
$\mathcal{P}(h^c)=0$. So, $\mathcal{P}\geq 0$ and this implies
$$
2\tau(p) \leq rC+\frac{(n-1)(n+r)(n^2-n-r)}{rn} C(L)+(n+9)(n-1).
$$
Therefore, we get
\begin{equation}\label{eq23}
\rho \leq
\frac{r}{n(n-1)}C+\frac{(n+r)(n^2-n-r)}{rn^2}C(L)+\frac{n+9}{n}
\end{equation}
for any hyperplane $L$ of $T_p M$ so that both inequalities
(\ref{eq17}) and (\ref{eq18}) easily follow from (\ref{eq23}).

Furthermore, we see that the equalities hold at the relations
(\ref{eq17}) and (\ref{eq18}) if and only if
\begin{align*}
&h_{ij}=0 \quad \text{for} \ i,j\in \{ 1,2,\cdots,n \} \ \text{with} \ i\neq j,  \\
&h_{nn}=\frac{n(n-1)}{r} h_{11} = \frac{n(n-1)}{r} h_{22} = \cdots =
\frac{n(n-1)}{r} h_{n-1n-1}.
\end{align*}
Therefore, we get that the equalities hold at (\ref{eq17}) and
(\ref{eq18}) if and only if the submanifold $M$ is invariantly
quasi-umbilical with flat normal connection in
$G_2(\mathbb{C}^{m+2})$ such that the shape operator takes the form
(\ref{eq188}) with respect to some orthonormal tangent and normal
frames.
\end{proof}

In the same way, by using (\ref{comm2}) and (\ref{curv2}), we obtain

\begin{theorem}
Let $M$ be a real hypersurface of a complex hyperbolic two-plane
Grassmannian $SU_{2,m}/S(U_2\cdot U_m)$ with $n=4m-1$. Then we have

(a) The generalized normalized $\delta$-Casorati curvature $\delta_c
(r,n-1)$ satisfies
\begin{equation}\label{eq24}
\rho \leq \frac{\delta_c (r,n-1)}{n(n-1)} - \frac{n+9}{2n}
\end{equation}
for any $r\in \mathbb{R}$ with $0<r<n(n-1)$.

(b) The generalized normalized $\delta$-Casorati curvature
$\widehat{\delta}_c (r,n-1)$ satisfies
\begin{equation}\label{eq25}
\rho \leq \frac{\widehat{\delta}_c (r,n-1)}{n(n-1)} - \frac{n+9}{2n}
\end{equation}
for any $r\in \mathbb{R}$ with $r>n(n-1)$.

Moreover, the equalities hold in the relations (\ref{eq24}) and
(\ref{eq25}) if and only if $M$ is an invariantly quasi-umbilical
submanifold with flat normal connection in $SU_{2,m}/S(U_2\cdot
U_m)$ such that with some orthonormal tangent frame $\{ e_1, \cdots,
e_n \}$ of $TM$ and orthonormal normal frame $\{ e_{n+1}=e \}$ of
$TM^{\perp}$, the shape operator $A_e$ takes the following form
$$
A_e = \left(
        \begin{array}{ccccc}
          a & 0 & \cdots & 0 & 0 \\
          0 & a & \cdots & 0 & 0 \\
          \vdots & \vdots & \ddots & \vdots & \vdots \\
          0 & 0 & \cdots & a & 0 \\
          0 & 0 & \cdots & 0 & \frac{n(n-1)}{r}a \\
        \end{array}
      \right).
$$
\end{theorem}

Using the relations $[\delta_c (\frac{n(n-1)}{2},n-1)](p) =
n(n-1)[\delta_c (n-1)](p)$ and $[\widehat{\delta}_c
(2n(n-1),n-1)](p) = n(n-1)[\widehat{\delta}_c (n-1)](p)$ for $p\in
M$, we easily have

\begin{corollary}
Let $M$ be a real hypersurface of a complex two-plane Grassmannians
$G_2(\mathbb{C}^{m+2})$ with $n=4m-1$. Then we get

(a) The normalized $\delta$-Casorati curvature $\delta_c (n-1)$
satisfies
\begin{equation}\label{eq26}
\rho \leq \delta_c (n-1) + \frac{n+9}{n}.
\end{equation}
Moreover, the equality holds if and only if $M$ is an invariantly
quasi-umbilical submanifold with flat normal connection in
$G_2(\mathbb{C}^{m+2})$ such that with some orthonormal tangent
frame $\{ e_1, \cdots, e_n \}$ of $TM$ and orthonormal normal frame
$\{ e_{n+1}=e \}$ of $TM^{\perp}$, the shape operator $A_e$ takes
the following form
$$
A_e = \left(
        \begin{array}{ccccc}
          a & 0 & \cdots & 0 & 0 \\
          0 & a & \cdots & 0 & 0 \\
          \vdots & \vdots & \ddots & \vdots & \vdots \\
          0 & 0 & \cdots & a & 0 \\
          0 & 0 & \cdots & 0 & 2a \\
        \end{array}
      \right).
$$
(b) The normalized $\delta$-Casorati curvature $\widehat{\delta}_c
(n-1)$ satisfies
\begin{equation}\label{eq27}
\rho \leq \widehat{\delta}_c (n-1) + \frac{n+9}{n}.
\end{equation}
Moreover, the equality holds if and only if $M$ is an invariantly
quasi-umbilical submanifold with flat normal connection in
$G_2(\mathbb{C}^{m+2})$ such that with some orthonormal tangent
frame $\{ e_1, \cdots, e_n \}$ of $TM$ and orthonormal normal frame
$\{ e_{n+1}=e \}$ of $TM^{\perp}$, the shape operator $A_e$ takes
the following form
$$
A_e = \left(
        \begin{array}{ccccc}
          2a & 0 & \cdots & 0 & 0 \\
          0 & 2a & \cdots & 0 & 0 \\
          \vdots & \vdots & \ddots & \vdots & \vdots \\
          0 & 0 & \cdots & 2a & 0 \\
          0 & 0 & \cdots & 0 & a \\
        \end{array}
      \right).
$$
\end{corollary}

\begin{corollary}
Let $M$ be a real hypersurface of a complex hyperbolic two-plane
Grassmannian $SU_{2,m}/S(U_2\cdot U_m)$ with $n=4m-1$. Then we
obtain

(a) The normalized $\delta$-Casorati curvature $\delta_c (n-1)$
satisfies
\begin{equation}\label{eq28}
\rho \leq \delta_c (n-1) - \frac{n+9}{2n}.
\end{equation}
Moreover, the equality holds if and only if $M$ is an invariantly
quasi-umbilical submanifold with flat normal connection in
$SU_{2,m}/S(U_2\cdot U_m)$ such that with some orthonormal tangent
frame $\{ e_1, \cdots, e_n \}$ of $TM$ and orthonormal normal frame
$\{ e_{n+1}=e \}$ of $TM^{\perp}$, the shape operator $A_e$ takes
the following form
$$
A_e = \left(
        \begin{array}{ccccc}
          a & 0 & \cdots & 0 & 0 \\
          0 & a & \cdots & 0 & 0 \\
          \vdots & \vdots & \ddots & \vdots & \vdots \\
          0 & 0 & \cdots & a & 0 \\
          0 & 0 & \cdots & 0 & 2a \\
        \end{array}
      \right).
$$
(b) The normalized $\delta$-Casorati curvature $\widehat{\delta}_c
(n-1)$ satisfies
\begin{equation}\label{eq29}
\rho \leq \widehat{\delta}_c (n-1) - \frac{n+9}{2n}.
\end{equation}
Moreover, the equality holds if and only if $M$ is an invariantly
quasi-umbilical submanifold with flat normal connection in
$SU_{2,m}/S(U_2\cdot U_m)$ such that with some orthonormal tangent
frame $\{ e_1, \cdots, e_n \}$ of $TM$ and orthonormal normal frame
$\{ e_{n+1}=e \}$ of $TM^{\perp}$, the shape operator $A_e$ takes
the following form
$$
A_e = \left(
        \begin{array}{ccccc}
          2a & 0 & \cdots & 0 & 0 \\
          0 & 2a & \cdots & 0 & 0 \\
          \vdots & \vdots & \ddots & \vdots & \vdots \\
          0 & 0 & \cdots & 2a & 0 \\
          0 & 0 & \cdots & 0 & a \\
        \end{array}
      \right).
$$
\end{corollary}

%\section*{Acknowledgments}

%The authors are grateful to the referees for their valuable comments and suggestions.

\end{document}